\documentclass{amsart}

\newtheorem{thm}{Theorem}[section]
\newtheorem{cor}[thm]{Corollary}

\newtheorem{prop}[thm]{Proposition}
\theoremstyle{definition}
\newtheorem{defn}[thm]{Definition}
\theoremstyle{remark}
\newtheorem{rem}[thm]{Remark}
\newtheorem*{ex}{Example}

\begin{document}
\title[3-submersions from QR-hypersurfaces]{\bf 3-submersions from QR-hypersurfaces of quaternionic K\"{a}hler manifolds}
\author{Gabriel Eduard V\^\i lcu}
\date{}
\maketitle
\abstract

In this paper we study 3-submersions from a QR-hypersurface of a
quaternionic K\"{a}hler manifold onto an almost quaternionic
hermitian manifold. We also prove the non-existence of quaternionic
submersions between quaternionic K\"{a}hler manifolds which are not
locally hyper-K\"{a}hler.

\endabstract

\section{Introduction}

In \cite{WTS} B. Watson introduced the notion of 3-submersion, as a
Riemannian submersion from an almost contact metric manifold onto an
almost quaternionic manifold, which commutes with the structure
tensors of type (1,1). In \cite{IMV1} and \cite{IMV2}, this concept
has been extended in quaternionic setting. In this paper we study
3-submersions from QR-hypersurfaces of quaternionic K\"{a}hler
manifolds, we give an example and obtain some obstructions to the
existence of quaternionic submersions.

The study of QR-submanifolds of a quaternionic K\"{a}hler manifold
was initiated by A. Bejancu \cite{BJC}. Among all submanifolds of a
quaternionic K\"{a}hler manifold, QR-submanifolds have been
intensively studied by several authors
\cite{AG,BEJ,BJC2,GSK,KP,KPK,MNG,SHN,SHN2}. In Section 2 one recalls
the definitions and basic properties of quaternionic manifolds and
QR-submanifolds of a quaternionic K\"{a}hler manifold.

On another hand, R. G\"{u}ne\c s, B. \c Sahin and S. Kele\c s
\cite{GSK} have shown that a QR-submanifold admits an almost contact
3-structure in some conditions. In Section 3 we see that on an
orientable hypersurface of a quaternionic K\"{a}hler manifold there
exists a natural almost contact metric 3-structure. This result will
allow us to define the concept of QR 3-submersion. In Section 4 we
obtain some properties for this kind of submersions and give an
example. In the last section we prove the non-existence of
quaternionic submersions between quaternionic K\"{a}hler non-locally
hyper-K\"{a}hler manifolds.

\section{Preliminaries}

Let $M$ be a differentiable manifold of dimension $n$ and assume
that there is a rank 3-subbundle $\sigma$ of $End(TM)$ such that a
local basis $\lbrace{J_1,J_2,J_3}\rbrace$ exists on sections of
$\sigma$ satisfying:
       \begin{equation}\label{1}
       \left\{\begin{array}{rcl}
       J_\alpha^2=-Id,\forall
       \alpha=\overline{1,3}\\
       J_1J_2=-J_2J_1=J_3
       \end{array}\right.
       \end{equation}

Then the bundle $\sigma$ is called an almost quaternionic structure
on $M$ and $\lbrace{J_1,J_2,J_3}\rbrace$ is called a canonical local
basis of $\sigma$. Moreover, $(M,g)$ is said to be an almost
quaternionic manifold. It is easy to see that any almost
quaternionic manifold is of dimension $n=4m$.

A Riemannian metric $g$ is said to be adapted to the quaternionic
structure $\sigma$ if it satisfies:
 \begin{equation}\label{2}
         g(J_\alpha X,J_\alpha Y)=
         g(X,Y),\forall\alpha=\overline{1,3}
         \end{equation}
for all vector fields $X$,$Y$ on $M$ and any local basis
$\lbrace{J_1,J_2,J_3}\rbrace$ of $\sigma$. Moreover, $(M,\sigma,g)$
is said to be an almost quaternionic hermitian manifold.

If the bundle $\sigma$ is parallel with respect to the Levi-Civita
connection $\nabla$ of $g$, then $(M,\sigma,g)$ is said to be a
quaternionic K\"{a}hler manifold. Equivalently, locally defined
1-forms $\omega_1,\omega_2,\omega_3$ exist such that:
    \begin{equation}\label{3}
          \left\{\begin{array}{rcl}
                 \nabla_XJ_1=\omega_3(X)J_2-\omega_2(X)J_3 \\
                 \nabla_XJ_2=-\omega_3(X)J_1+\omega_1(X)J_3 \\
                 \nabla_XJ_3=\omega_2(X)J_1-\omega_1(X)J_2
                 \end{array}\right.
    \end{equation}
for any vector field $X$ on $M$. In particular, if
$\omega_1=\omega_2=\omega_3=0$, then $(M,\sigma,g)$ is said to be a
locally hyper-K\"{a}hler manifold.

We remark that any quaternionic K\"{a}hler manifold $M$ is an
Einstein manifold, provided that $dim M>4$. Moreover, $M$ is
irreducible (if Ric$\neq$ 0) or locally hyper-K\"{a}hler manifold
(if Ric=0) (see \cite{AL1,BES,ISH,SLM}).

Let $(M,\sigma,g)$ be an almost quaternionic hermitian manifold. If
$X\in T_pM, p\in M$, then the 4-plane $Q(X)$ spanned by $\lbrace
{X,J_1X,J_2X,J_3X}\rbrace$ is called a quaternionic 4-plane. A
2-plane in $T_pM$ spanned by $\lbrace{X,Y}\rbrace$ is called
half-quaternionic if $Q(X)=Q(Y)$.

The sectional curvature for a half-quaternionic 2-plane is called
quaternionic sectional curvature. A quaternionic K\"{a}hler manifold
is a quaternionic space form if its quaternionic sectional
curvatures are equal to a constant, say $c$. It is well-known that a
quaternionic K\"{a}hler manifold $(M,\sigma,g)$ is a quaternionic
space form (denoted $M(c)$) if and only if its curvature tensor is:
     \begin{eqnarray}\label{4}
       R(X,Y)Z&=&\frac{c}{4}\lbrace g(Z,Y)X-
       g(X,Z)Y+\sum\limits_{\alpha=1}^3
        [g(Z,J_\alpha Y)J_\alpha X\nonumber\\
       &&-g(Z,J_\alpha X)J_\alpha Y+
        2g(X,J_\alpha Y)J_\alpha Z]\rbrace
      \end{eqnarray}
for all vector fields $X,Y,Z$ on $M$ and any local basis
$\lbrace{J_1,J_2,J_3}\rbrace$ of $\sigma.$

Let $(\overline M,\sigma,g)$ be a quaternionic K\"{a}hler manifold
and let $M$ be a real submanifold of $\overline M$. Then $M$ is
called QR-submanifold if there exists a vector subbundle $D$ of the
normal bundle $TM^\perp$ such that we have:

i. $J_\alpha(D_p)=D_p, \forall p\in M,
\forall \alpha=\overline {1,3}$;

ii. $J_\alpha(D_p^\perp)\subset T_pM, \forall p\in M, \forall
\alpha=\overline {1,3}$, where $D^\perp$ is the complementary
orthogonal bundle to $D$ in $TM^\perp$ (see \cite{BJC}).

\section{QR-hypersurfaces and almost contact metric 3-structures}

Let $M$ be an orientable hypersurface of a quaternionic K\"{a}hler
manifold $\overline M$ and $\xi$ a unit normal field on $M$. If we
take $D={0}$, then $D^\perp=TM^\perp$ and we conclude that $M$ is a
QR-submanifold of $\overline M$.

Let $\{J_\alpha\}_{\alpha=\overline{1,3}}$ and
$\{J'_\alpha\}_{\alpha=\overline{1,3}}$ two local bases defined on
coordinate neighborhoods $\overline U$ and $\overline U'$, with
$\overline U\bigcap\overline U'\neq\emptyset$. Then, on $\overline
U$:
$$\xi_\alpha=-J_\alpha\xi, \forall \alpha=\overline{1,3},$$ defines
tangent vector fields to $M$ and similarly, on $\overline U'$:
$$\xi'_\alpha=-J'_\alpha\xi, \forall \alpha=\overline{1,3},$$ defines
tangent vector fields to $M$.

Moreover, on $\overline U\bigcap\overline U'$ we have:
$$\xi'_\alpha=\sum\limits_{\beta=1}^3c_{\alpha\beta}\xi_\beta, \forall \alpha=\overline{1,3},$$
where $C=(c_{\alpha\beta})_{\alpha,\beta=\overline{1,3}}\in SO(3)$.
Thus, we obtain a distribution $\mathcal{V}$ on $M$, which is
locally generated by $\{\xi_\alpha\}_{\alpha=\overline{1,3}}$. Let
$\mathcal{H}$ be the orthogonal complementary distribution to
$\mathcal{V}$ with respect to the Riemannian metric $g$ induced by
$\overline{g}$ on $M$. We remark that for each $p\in M$,
$\mathcal{H}_p$ is $J_\alpha$-invariant, $\forall
\alpha=\overline{1,3}$.

We recall that the distribution $\mathcal{V}$ is integrable if and
only if $M$ is a mixed geodesic QR-hypersurface of $\overline M$,
i.e:
\begin{equation}
         B(U,X)=0, \forall U\in\Gamma(\mathcal{V}), \forall
         X\in\Gamma(\mathcal{H}),
         \end{equation}
where $B$ is the second fundamental form of $M$ in $\overline M$
(see \cite{BJC}).

\begin{defn}  \cite{BLR}
        Let $M$ be a differentiable manifold equipped with a triple
        $(\phi,\xi,\eta)$, where $\phi$ is a field  of endomorphisms
        of the tangent spaces, $\xi$ is a vector field and $\eta$ is a
        1-form on $M$. If we have:
\begin{equation}\label{18a}
        \phi^2=-I+\eta\otimes\xi,\ \  \eta(\xi)=1
\end{equation}
        then we say that $(\phi,\xi,\eta)$ is an almost contact
        structure on $M$.
\end{defn}

\begin{defn} \cite{KUO}
        Let $M$ be a differentiable manifold which admits three almost contact structure
        $(\phi_\alpha,\xi_\alpha,\eta_\alpha), \forall
\alpha=\overline{1,3}$ satisfying the following
        conditions:
        \begin{equation}\label{19}
        \eta_\alpha(\xi_\beta)=0, \forall \alpha\neq\beta,
        \end{equation}
        \begin{equation}\label{20}
        \phi_\alpha(\xi_\beta)=-\phi_\beta(\xi_\alpha)=\xi_\gamma,
        \end{equation}
        \begin{equation}\label{21}
        \eta_\alpha\circ\phi_\beta=-\eta_\beta\circ\phi_\alpha=\eta_\gamma
        \end{equation}
        and
        \begin{equation}\label{22}
        \phi_\alpha\phi_\beta-\eta_\beta\otimes\xi_\alpha=
        -\phi_\beta\phi_\alpha+\eta_\alpha\otimes\xi_\beta=\phi_\gamma,
        \end{equation}
        where in (\ref{20}), (\ref{21}) and (\ref{22}), $(\alpha,\beta,\gamma)$ is an even
        permutation of (1,2,3).
        Then the manifold $M$ is said to have an almost contact 3-structure
        $(\phi_\alpha,\xi_\alpha,\eta_\alpha)_{\alpha=\overline{1,3}}$.
\end{defn}
\begin{defn} \cite{KUO}
        Let $(M,g)$ be a Riemannian manifold, endowed with an
        almost  contact 3-structure
        $(\phi_\alpha,\xi_\alpha,\eta_\alpha)_{\alpha=\overline{1,3}}$
        such that:
        \begin{equation}\label{23}
        \eta_\alpha(X)=g(X,\xi_\alpha), \forall
\alpha=\overline{1,3}
        \end{equation}
        and
         \begin{equation}\label{24}
        g(\phi_\alpha X,\phi_\alpha
        Y)=g(X,Y)-\eta_\alpha(X)\eta_\alpha(Y), \forall
\alpha=\overline{1,3}
        \end{equation}
for all vector fields $X,Y$ on $M$. Then we say that $M$ admits an
almost contact metric 3-structure.
\end{defn}
\begin{defn} \cite{BLR}
An almost contact metric 3-structure
$(\phi_\alpha,\xi_\alpha,\eta_\alpha)_{\alpha=\overline{1,3}}$ on a
Riemannian manifold $(M,g)$ is said to be a 3-cosymplectic structure
if:
     \begin{equation}\label{25}
        (\nabla_X \phi_\alpha)(Y)=0, (\nabla_X \eta_\alpha)(Y)=0,  \forall
\alpha=\overline{1,3}.
        \end{equation}
\end{defn}

Let $M$ be an orientable hypersurface of a quaternionic K\"{a}hler
manifold $\overline M$. If $S:TM\rightarrow\mathcal{H}$ is the
canonical projection, then, we have that any local vector field $X$
on $M$ is expressed as follows:
     \begin{equation}\label{27}
       X=SX+\sum_{\alpha=1}^{3}\eta_{\alpha}(X)\xi_{\alpha},
       \end{equation}
where:
      \begin{equation}\label{28}
       \eta_{\alpha}(X)=g(X,\xi_{\alpha}), \forall
\alpha=\overline{1,3}.
       \end{equation}

From (\ref{27}) we have:
     \begin{equation}\label{29}
       J_\alpha X=J_\alpha
       SX+\sum_{\beta=1}^{3}\eta_{\beta}(X)J_\alpha\xi_{\beta}, \forall
\alpha=\overline{1,3}.
       \end{equation}

From (\ref{29}) we obtain the decomposition:
      \begin{equation}\label{30}
       J_\alpha X=\phi_\alpha X+F_\alpha X,
       \end{equation}
where $\phi_\alpha X$ is the tangential part of $J_\alpha X$, given
by:
       \begin{equation}\label{31}
       \phi_\alpha X=J_\alpha SX+\eta_\beta(X)\xi_\gamma-\eta_\gamma(X)\xi_\beta,
       \end{equation}
and $F_\alpha X$ is the normal part of $J_\alpha X$, given by:
       \begin{equation}\label{32}
       F_\alpha X=\eta_\alpha(X)\xi,
       \end{equation}
for all $\alpha=\overline{1,3}$, where $(\alpha,\beta,\gamma)$ is an
even permutation of (1,2,3).

By straightforward computations, we can easily see that
$(\phi_\alpha,\xi_\alpha,\eta_\alpha)_{\alpha}$, defined by
(\ref{28}), (\ref{31}) and(\ref{32}), is an almost contact metric
3-structure on $M$ (see also \cite{GSK}) and so we have the next
result.
\begin{prop} \label{3.1}
Any QR-hypersurface of a quaternionic K\"{a}hler manifold admits a
natural almost contact metric 3-structure.
\end{prop}

\section{3-submersions of QR-hypersurfaces}

\begin{defn}\label{4.1}
Let $M$ be a mixed geodesic QR-hypersurface of a quaternionic
K\"{a}hler manifold $\overline M$, endowed with the natural almost
contact metric 3-structure
$(\phi_\alpha,\xi_\alpha,\eta_\alpha)_{\alpha=\overline{1,3}}$,
given by Proposition \ref{3.1} and let $(M',\sigma',g')$ be an
almost quaternionic hermitian manifold. We say that a Riemannian
submersion $\pi:M\rightarrow M'$ is a QR 3-submersion if the
following conditions are satisfied:

i. $Ker\pi_*=\mathcal{V}$;

ii. For each $p\in M$, $\sigma'_{\pi(p)}$ admits a canonical local
basis $\lbrace{J'_1,J'_2,J'_3}\rbrace$ such that:
$$\pi_*\phi_\alpha=J'_\alpha\pi_*, \forall \alpha=\overline{1,3}.$$
\end{defn}

\begin{rem}\label{4.2a}
We recall that the sections of $\mathcal{V}$, respectively
$\mathcal{H}$, are called the vertical vector fields, respectively
horizontal vector fields. A Riemannian submersion $\pi:M\rightarrow
M'$ determines two (1,2) tensor field $T$ and $A$ on $M$, by the
formulas:
\begin{equation}\label{33}
       T(E,F)=T_EF=h\nabla_{vE}vF+v\nabla_{vE}hF
       \end{equation}
and respectively:
\begin{equation}\label{34}
       A(E,F)=A_EF=v\nabla_{hE}hF+h\nabla_{hE}vF
       \end{equation}
for any $E,F\in \Gamma(TM)$, where $v$ and $h$ are the vertical and
horizontal projection (see \cite{KO,ON}).

We remark that for $U,V\in\Gamma(\mathcal{V})$, $T_U V$ coincides
with the second fundamental form of the immersion of the fibre
submanifolds and for $X,Y\in\Gamma(\mathcal{H})$, $A_X
Y=\frac{1}{2}v[X,Y]$ reflecting the complete integrability of the
horizontal distribution $\mathcal{H}$.

An horizontal vector field $X$ on $M$ is said to be basic if $X$ is
$\pi$-related to a vector field $X'$ on $M'$. It is clear that every
vector field $X'$ on $M'$ has a unique horizontal lift $X$ to $M$
and $X$ is basic.
\end{rem}
\begin{rem} \label{4.2}
If $\pi:M\rightarrow M'$ is a Riemannian submersion and $X,Y$ are
basic vector fields on $M$, $\pi$-related to $X'$ and $Y'$ on $M'$,
then we have the next properties (see \cite{BES,FIP,ON}):

i. $h[X,Y]$ is a basic vector field and
$\pi_*h[X,Y]=[X',Y']\circ\pi$;

ii. $h(\nabla_XY)$ is a basic vector field $\pi$-related to
$\nabla'_{X'}Y'$, where $\nabla$ and $\nabla'$ are the Levi-Civita
connections on $M$ and $M'$;

iii. $[E,U]\in\Gamma(\mathcal{V}), \forall U\in\Gamma(\mathcal{V})$
and $\forall E\in\Gamma(TM)$.
\end{rem}

\begin{prop}
Let $M$ be a mixed geodesic QR-hypersurface of a quaternionic
K\"{a}hler manifold $(\overline M,\overline\sigma,\overline g)$ and
let $(M',\sigma',g')$ be an almost quaternionic hermitian manifold.
If $\pi:M\rightarrow M'$ is a QR 3-submersion, then the
distributions $\mathcal{V}$ and $\mathcal{H}$ are invariant by
$\phi_\alpha, \forall \alpha=\overline{1,3}$.
\end{prop}
\begin{proof}
Let $V\in\Gamma(\mathcal{V})$. Then, we have:
$$\pi_*\phi_\alpha V=J'_\alpha \pi_* V=0,$$
and so we conclude that
$\phi_\alpha(\mathcal{V})\subset\mathcal{V}$.

On the other hand, for any $X\in\Gamma(\mathcal{H})$, we have:
$$g(\phi_\alpha X,V)=g(J_\alpha X,V)=-g(X,J_\alpha V)=0,$$
and thus we obtain $\phi_\alpha(\mathcal{H})\subset\mathcal{H}$.
\end{proof}

\begin{thm}\label{4.5}
Let  $\pi:M\rightarrow M'$ be a QR 3-submersion such that the
canonical almost contact 3-structure
$(\phi_\alpha,\xi_\alpha,\eta_\alpha)_{\alpha=\overline{1,3}}$ on
$M$ is a 3-cosymplectic structure. Then $M'$ is locally
hyper-K\"{a}hler.
\end{thm}
\begin{proof}
For any local basic vector fields $X,Y$ on $M$, $\pi$-related with
$X'$ and $Y'$ on $M'$, from (\ref{25}) we have:
       \begin{equation}\label{36}
       \nabla_X\phi_\alpha Y-\phi_\alpha\nabla_XY=0, \forall
\alpha=\overline{1,3}.
\end{equation}
and from (\ref{36})  we deduce:
\begin{equation}\label{37}
       \pi_*(\nabla_{X}\phi_\alpha Y)-\pi_*\phi_\alpha\nabla_XY=0, \forall
\alpha=\overline{1,3}.
\end{equation}
Then, since $Y$ is a basic vector field $\pi$-related with $Y'$,
also $\phi_\alpha Y$ is basic and $\pi$-related with $J'_\alpha Y'$
and taking account of Definition \ref{4.1} and Remark \ref{4.2}, we
obtain from (\ref{37}) that we have:
$$\nabla'_{X'}J'_\alpha Y'-J'_\alpha\nabla'_{X'}Y'=0, \forall
\alpha=\overline{1,3},$$ and thus we conclude that
$(\nabla'_{X'}J'_\alpha) Y'=0$, and so $M'$ is locally
hyper-K\"{a}hler.
\end{proof}

\begin{cor}
Let $M$ be a totally geodesic QR-hypersurface of a quaternionic
K\"{a}hler manifold $(\overline M,\overline \sigma,\overline g)$ and
$(M',\sigma',g')$ be an almost quaternionic hermitian manifold. If
$\pi:M\rightarrow M'$ is a QR 3-submersion such that $\xi_1,\xi_2$
and $\xi_3$ are parallel in $M$, then $M'$ is locally
hyper-K\"{a}hler.
\end{cor}
\begin{proof}
In this case
$(\phi_\alpha,\xi_\alpha,\eta_\alpha)_{\alpha=\overline{1,3}}$ is a
3-cosymplectic structure on $M$ (see \cite{GSK}) and the proof is
obvious from Theorem \ref{4.5}.
\end{proof}

\begin{thm}
Let $M$ be a mixed geodesic QR-hypersurface of a quaternionic
K\"{a}hler manifold $(\overline M,\overline \sigma,\overline g)$,
$(M',\sigma',g')$ be an almost quaternionic hermitian manifold and
$\pi:M\rightarrow M'$ be a QR 3-submersion. If the natural almost
contact metric 3-structure
$(\phi_\alpha,\xi_\alpha,\eta_\alpha)_{\alpha=\overline{1,3}}$ on
$M$ is 3-cosymplectic, then the fibre submanifolds are totally
geodesic immersed and the horizontal distribution is integrable.
\end{thm}
\begin{proof}
Since $M$ is 3-cosymplectic we have:
       \begin{equation}\label{38}
       \nabla_U\phi_\alpha V=\phi_\alpha\nabla_UV, \forall
       \alpha=\overline{1,3},
\end{equation}
for $\forall U,V\in\Gamma (\mathcal{V})$. Taking the horizontal
components, we obtain:
       \begin{equation}\label{39}
       T_U\phi_\alpha V=\phi_\alpha T_UV, \forall
       \alpha=\overline{1,3}
\end{equation}
which immediately implies:
       \begin{equation}\label{41}
       T_U V=-T_{\phi_\alpha U} \phi_\alpha V, \forall
       \alpha=\overline{1,3}.
\end{equation}

From (\ref{41}), taking account of (\ref{1}), we obtain $T=0$.
Similarly we obtain $A=0$ and the proof is now complete, via Remark
\ref{4.2a}.
\end{proof}

Let $M$ be an orientable submanifold of a Riemannian manifold
$(\overline{M},\overline{g})$. We say that $M$ is a totally
umbilical submanifold of $\overline{M}$ if the second fundamental
form $h$ of $M$ satisfied:
\begin{equation}\label{44}
       h(E,F)=g(E,F)H
       \end{equation}
$\forall E,F \in \Gamma(TM)$, where $H$ is the mean curvature vector
field on $M$. Moreover, if $H$ is nonzero and parallel in the normal
bundle $TM^\perp$, then $M$ is called an extrinsic sphere.

By using the Gauss equation, (\ref{4}) and the Gray-O'Neill equation
(see \cite{BES,FIP,MNG,ON}), we can easily prove the next result.

\begin{thm}
Let $M$ be a QR extrinsic hypersphere of a flat quaternionic
K\"{a}hler manifold $(\overline M,\overline \sigma,\overline g)$ and
let $(M',\sigma',g')$ be an another quaternionic K\"{a}hler
manifold. If $\pi:M\rightarrow M'$ is a QR 3-submersion, then $M'$
is a quaternionic space form.
\end{thm}

\begin{ex}
Let $S^{4m+3}$ be the standard hypersphere in $R^{4m+4}$. Then the
canonical mapping $\pi:S^{4m+3}\rightarrow P^m(H)$ is a QR
3-submersion.
\end{ex}

\section{Quaternionic submersions}

\begin{defn} \cite{IMV1}
Let $(M,\sigma,g)$ and $(N,\sigma',g')$ be two almost quaternionic
hermitian manifolds. A map $f:M\rightarrow N$ is said to be a
$(\sigma,\sigma')$-holomorphic map at a point $x\in M$ if for any
$J\in\sigma_x$ exists $J'\in\sigma'_{f(x)}$ such that $f_*\circ
J=J'\circ f_*$. Moreover, we say that $f$ is a
$(\sigma,\sigma')$-holomorphic  map if $f$ is a
$(\sigma,\sigma')$-holomorphic  map at each point $x\in M$.
\end{defn}

\begin{defn} \cite{IMV2}
Let $(M,\sigma,g)$ and $(N,\sigma',g')$ be two almost quaternionic
hermitian manifolds. A Riemannian submersion $\pi:M\rightarrow N$
which is a $(\sigma,\sigma')$-holomorphic  map is called a
quaternionic submersion.
\end{defn}

\begin{thm}
Let $\pi:(M,\sigma,g)\rightarrow(N,\sigma',g')$ be a quaternionic
submersion such that $(M,\sigma,g)$ is a quaternionic K\"{a}hler
manifold. Then $(N,\sigma',g')$ is a quaternionic K\"{a}hler
manifold.
\end{thm}
\begin{proof}
If we consider $X_*,Y_*\in\Gamma(TN)$ such that $\pi_*X=X_*,
\pi_*Y=Y_*$, where $X,Y\in\Gamma(TM)$, we obtain:
    \begin{eqnarray}\label{20b}
    (\nabla'_{X_*}J'_\alpha)Y_*&=&\nabla'_{X_*}(J'_\alpha Y_*)-J'_\alpha(\nabla'_{X_*}Y_*)\nonumber\\
    &&=\nabla'_{\pi_*X}(J'_\alpha \pi_*Y)-J'_\alpha(\nabla'_{\pi_*X}\pi_*Y)\nonumber\\
    &&=\nabla'_{\pi_*X}(\pi_*(J_\alpha Y))-J'_\alpha\pi_*(h\nabla_X Y)\nonumber\\
    &&=\pi_*(h\nabla_X(J_\alpha Y))-\pi_*(J_\alpha(h\nabla_X Y))\nonumber\\
    &&=\pi_*((\nabla_X J_\alpha)Y).
    \end{eqnarray}

Since $(M,\sigma,g)$ is a quaternionic K\"{a}hler manifold we have
(\ref{3}) and we can define 1-forms $\omega'_1,\omega'_2,\omega'_3$
on $N$ by:
    \begin{equation}\label{21b}
    \omega'_\alpha(X_*)\circ \pi=\omega_\alpha(X), \forall
    \alpha\in\{1,2,3\},
    \end{equation}
for any local vector field $X_*$ on $N$ and $X$ a real basic vector
field on $M$ such that $\pi_*X=X_*$.

From (\ref{3}), (\ref{20b}) and (\ref{21b}) we deduce for all
$\alpha\in\{1,2,3\}$:
    \begin{equation}\label{22b}
    (\nabla'_{X_*}J'_{\alpha})Y_*=\omega'_{\alpha+2}(X_*)J'_{\alpha+1}Y_*-
    \omega'_{\alpha+1}(X_*)J'_{\alpha+2}Y_*,
    \end{equation}
for any local vector fields $X_*,Y_*$ on $M'$, where the indices are
taken from $\{1,2,3\}$ modulo 3. Thus we conclude that
$(N,\sigma',g')$ is a quaternionic K\"{a}hler manifold.
\end{proof}

\begin{cor}
Let $\pi:(M,\sigma,g)\rightarrow(N,\sigma',g')$ be a quaternionic
submersion such that $(M,\sigma,g)$ is a quaternionic K\"{a}hler
manifold. Then, both $(M,\sigma,g)$ and $(N,\sigma',g')$ are locally
hyper-K\"{a}hler manifolds.
\end{cor}
\begin{proof}
In this case we have that the vertical and horizontal distributions
are both integrable (see \cite{IMV2}) and so we can easily conclude
that $(M,\sigma,g)$ is a locally hyper-K\"{a}hler manifold. The
assertion follows now from the above Theorem.
\end{proof}

\begin{cor}
There are no quaternionic submersions between quaternionic
K\"{a}hler manifold which are not locally hyper-K\"{a}hler.
\end{cor}
\begin{proof}
The assertion is obvious from the above Corollary.
\end{proof}

\section*{Acknowledgments}
The author expresses his gratitude to the referee for carefully
reading the manuscript and giving useful comments. This work was
partially supported by a PN2-IDEI grant, no. 525/2009.

\begin{center}
Gabriel Eduard V\^\i lcu$^{1,2}$ \\
         $^1${\em  University of Bucharest,\\
         Faculty of Mathematics and Computer Science,\\
         Research Center in Geometry, Topology and Algebra,\\
         Str. Academiei, Nr.14, Sector 1, Bucure\c sti, Romania}\\
         $^2${\em''Petroleum-Gas'' University of Ploie\c sti,\\
         Department of Mathematics and Computer Science,\\
         Bulevardul Bucure\c sti, Nr. 39, Ploie\c sti, Romania}\\
         e-mail: gvilcu@mail.upg-ploiesti.ro\\
\end{center}

\end{document}